\documentclass[a4paper,11pt]{amsart}

\title[Estimates of the amplitude of holonomies]{Estimates of the amplitude of holonomies by the curvature of a connection on a bundle}

\author{Sagun Chanillo}
\address{Dept. of Mathematics\\
Rutgers University\\
110 Frelinghuysen Rd.\\
Piscataway, NJ 08854\\
USA} \email{chanillo@math.rutgers.edu}

\author{Jean Van Schaftingen}
\address{Universit\'e catholique de Louvain\\
Institut de Recherche en Math\'ematique et Physique\\
Chemin du Cyclotron 2 bte L7.01.01\\
1348 Louvain-la-Neuve\\
Belgium} 
\email{Jean.VanSchaftingen@uclouvain.be} 

\dedicatory{To Ha\"im Brezis in friendship and admiration}

\usepackage[T1]{fontenc}
\usepackage[utf8]{inputenc}
\usepackage{geometry}
\usepackage{microtype}
\usepackage{lmodern}
\usepackage{stmaryrd}

\DeclareFontFamily{U}{MnSymbolC}{}
\DeclareSymbolFont{MnSyC}{U}{MnSymbolC}{m}{n}
\DeclareFontShape{U}{MnSymbolC}{m}{n}{
    <-6>  MnSymbolC5
   <6-7>  MnSymbolC6
   <7-8>  MnSymbolC7
   <8-9>  MnSymbolC8
   <9-10> MnSymbolC9
  <10-12> MnSymbolC10
  <12->   MnSymbolC12}{}
\DeclareMathSymbol{\intprod}{\mathbin}{MnSyC}{'270}

\usepackage{amssymb}
\usepackage{amsmath}
\usepackage{amsthm}
\usepackage{esint}

\usepackage[foot]{amsaddr}
\usepackage{constants}
\usepackage{paralist}
\usepackage{mathrsfs}
\newcommand{\defeq}{\triangleq}
\newcommand{\compose}{\,\circ\,}

\usepackage{todonotes}

\usepackage[pdftitle={Estimates of the amplitude of holonomies by the curvature of a principal connection},%
pdfauthor={Jean Van Schaftingen},%
final]{hyperref}

\usepackage[nameinlink%
,capitalize%
]{cleveref}

\usepackage[abbrev,backrefs]{amsrefs}

\newtheorem{proposition}{Proposition}[section]
\newtheorem{theorem}[proposition]{Theorem}
\newtheorem{lemma}[proposition]{Lemma}
\newtheorem{corollary}[proposition]{Corollary}

\theoremstyle{definition}

\theoremstyle{remark}

\numberwithin{equation}{section}

\newcommand{\abs}[1]{{\lvert #1 \rvert}}

\newcommand{\biggabs}[1]{{\biggl\lvert #1 \biggr\rvert}}

\newcommand{\st}{\;:\;}

\newcommand{\ampl}[1]{\langle\!\langle #1 \rangle\!\rangle}

\newcommand{\Rset}{\mathbb{R}}

\newcommand{\Cset}{\mathbb{C}}
\newcommand{\Sset}{\mathbb{S}}
\newcommand{\Bset}{\mathbb{B}}

\newcommand{\dif}{\,\mathrm{d}}

\newcommand{\eofs}{\,}

\allowdisplaybreaks

\begin{document}

\begin{abstract}
We show how the amplitude of holonomies on a vector bundle can be controlled by the integral of the curvature of the connection on a surface enclosed by the curve.
\end{abstract}

\maketitle

\section{Introduction}

Let \(E\) be a vector bundle over the manifold \(M\) with fiber \(F\) endowed with a connection \(\nabla\), that is in a local trivialization \(U \times F \subset F\),
\(\nabla = d + \omega\), where for each point \(p \in U\), \(\omega_p : U \to \operatorname{F} (T_pM \to \mathfrak{gl} (F))\) is the connection form, \(\operatorname{F} (T_p M \to \mathfrak{gl} (F))\) being the space of linear forms from the tangent space \(T_p M\) to the Lie algebra \(\mathfrak{gl} (F)\).
Any map \(\gamma \in C^1 ([0, 1], M)\) defines a parallel transport \(\operatorname{Pt}_\gamma : [0, 1] \to GL (E_p)\) from \(E_{\gamma (0)}\) to \(E_{\gamma (1)}\) as a solution to the problem \(\operatorname{Pt}_\gamma' (t) + \omega_{\gamma (t)}[\gamma' (t)] \operatorname{Pt}_{\gamma (t)} = 0\) and \(\operatorname{Pt}_\gamma (0) = \operatorname{id}\). If \(\gamma (0) = \gamma (1) = p \in M\), \(\operatorname{Pt}_{\gamma} \in GL (E_p)\) is the holonomy of the connection along \(\gamma\).

The holonomy group of the connection \(\nabla\) at \(p\) is the group generated by all the holonomies. A fundamental question is the relationship between the holonomy at a point and the curvature form of the connection which is represented in a local trivialization as
\begin{equation*}
 \Omega = d \omega + \omega \wedge \omega
 \eofs .
\end{equation*}

Algebraically, this is settled by the Ambrose and Singer theorem \cite{Ambrose_Singer_1953}, originating from \'E. Cartan's work \cite{Cartan_1926}*{p.~4} (see also \citelist{\cite{Kobayashi_Nomizu_1963}*{theorem II.8.1}\cite{Reckziegel_Wilhelmus_2006}*{theorem 2}\cite{Sternberg_1964}*{theorem 1.2}\cite{Nijenhuis}*{Theorem 1}}), which states that the identity component of the holonomy group at \(p \in M\) coincides with the group of holonomies along null-homotopic loops and that the corresponding Lie algebra is generated by the images of the curvature form at any point of the connected component of \(p\) in \(M\) and transported parallely at the point \(p\). In particular the Lie algebra corresponding to the normal closure of the holonomy group is generated by the values of the curvature form in all local trivializations.

\bigbreak

We consider here the quantitative corresponding question about how the holonomy can be controlled by the curvature.
More precisely, we assume that the structure group \(G \subset GL (F)\) is endowed with a bi-invariant metric and we define the holonomy amplitude of a curve \(\gamma \in C^1 ([0, 1], M)\) by
\begin{multline}
    \ampl{\gamma}
  =
    \inf
    \,
    \biggl\{
      \int_0^1 \abs{g'}
      \st
        g \in C^1 ([0, 1], G),
        g (0) = \operatorname{id},
        g (1) = \operatorname{Pt}_{\gamma} (1)\\[-1em]
        \text{ and \(g\) and \(\operatorname{Pt}_\gamma\) are homotopic
        relatively to \(\{0, 1\}\)}
    \biggr\}
    \eofs .
\end{multline}
The amplitude depends on the connection, the structure group \(G\) and the metric on \(G\), and is invariant under changes of gauge.
If \(G\) is simply connected (which can in fact always be assumed by replacing the group \(G\) by its universal covering), the amplitude  corresponds to the geodesic distance between the identity \(\operatorname{id}\) and \(\operatorname{Pt}_{\gamma} (b)\).

In the case where \(G\) is an abelian group, then the holonomy amplitude can be computed by the integral formula
\begin{equation*}
   \ampl{\gamma} 
 = 
   \biggabs{\int_0^1 \gamma^\ast  \omega},
\end{equation*}
where \(\gamma^\ast  \omega\) is the pull back of the differential form \(\omega\), defined for each \(t \in [a, b]\) by \(\gamma^\ast  \omega (t)[v] \defeq \omega (\gamma (t))[\gamma' (t)]\).
If \(\sigma \in C^1 (\Bset^2, M)\) and if \(\gamma : [0, 1] \to M\) is defined for \(t \in [0, 1]\) by \(\gamma (t) \defeq \sigma (\cos 2 \pi t, \sin 2 \pi t)\), we have by the Stokes--Cartan formula
\begin{equation}
\label{eq_iizahK3hai}
   \ampl{\gamma} =
   \biggabs{\int_{\Bset^2} \sigma^\ast  d \omega\,}
   =  \biggabs{\int_{\Bset^2} \sigma^\ast  \Omega \,}
   \eofs,
\end{equation}
since the group \(G\) is abelian and thus \(\Omega = d \omega\).
This implies the estimate,
\begin{equation}
  \ampl{\gamma} \le
  \int_{\sigma (\Bset^2)} \abs{\Omega} \dif \mathcal{H}^2
  \eofs,
\end{equation}
where the two-dimensional Hausdorff measure \(\mathcal{H}^2\) is taken with respect to a Riemannian metric on the manifold \(M\) and the norm with respect to the same Riemannian metric and with respect to the metric on the Lie algebra \(\mathfrak{g}\).
If \(M = \Rset^m\), by the isoperimetric inequality \cite{Almgren_1986} this implies that for every closed curve \(\gamma : [0, 1] \to \Rset^m\)
\begin{equation}
\label{eq_fieHiefai0}
  \ampl{\gamma} \le \frac{\operatorname{length} (\gamma)^2}{4 \pi} \sup_M \abs{\Omega}
  \eofs.
\end{equation}

When \(G = U (1) \simeq SO (2)\), the connections are related to electro-magnetic gauge theories and the curvature \(\Omega\) of the connection corresponds to the magnetic field.
Such connections appear in the definition of \emph{magnetic Sobolev spaces} \citelist{\cite{Esteban_Lions_1989}\cite{Lieb_Loss_2001}*{7.19--7.22}\cite{Kato_1972}*{(2.1)}}.
The analysis of magnetic Sobolev spaces should be invariant under gauge transformation, that is, it should not depend on a particular choice of a local trivialization.
In a recent work, Nguyen Hoai-Minh and the second author Jean Van Schaftingen have studied the problem of traces of magnetic Sobolev functions with constructions and estimates that depend only on the curvature of the connection \cite{Nguyen_VanSchaftingen}; a key point in this work was the estimate \eqref{eq_fieHiefai0} for \(U (1)\)--bundles.
A nonabelian gauge-invariant extension of the theory of magnetic Sobolev spaces requires thus new estimates on the holonomy amplitude.

We obtain the following non-abelian version of \eqref{eq_iizahK3hai}.

\begin{theorem}
\label{mainTheorem}
If \(\sigma \in C^1 (\Bset^2, M)\) and \(\gamma : [0, 1] \to M\) is defined for \(t \in [0, 1]\) by \(\gamma (t) \defeq \sigma (\cos 2 \pi t, \sin 2 \pi t)\), then
\begin{equation*}
  \ampl{\gamma} \le
  \int_{\Bset^2} \abs{\sigma^\ast  \Omega}
  \eofs.
\end{equation*}
\end{theorem}

Here, \(\sigma^\ast  \Omega\) is a \(\mathfrak{g}\)--valued \(2\)--form and \(\abs{\sigma^\ast  \Omega}\) is the associated \emph{density} \citelist{\cite{Loomis_Sternberg}*{\S 10.3}\cite{Nicolaescu_2007}*{\S 3.4.1}\cite{Folland_1999}*{\S 11.4}}.

\begin{corollary}
\label{corollaryEuclidean}
If \(M = \Rset^m\), if \(\gamma \in C^1([0, 1],  M)\) and if \(\gamma (0) = \gamma (1)\), then
\begin{equation*}
  \ampl{\gamma}
 \le
  \frac{\operatorname{length} (\gamma)^2 \, \sup_M  \abs{\Omega}}{4 \pi}
 \eofs.
\end{equation*}
\end{corollary}

\Cref{corollaryEuclidean} follows from \cref{mainTheorem} and from the observation that any closed curve \(\gamma\) bounds some minimal surface of area at most \(\frac{1}{4 \pi}(\operatorname{length} (\gamma))^2\) \cite{Almgren_1986}.

The proofs of \cref{mainTheorem} and \cref{corollaryEuclidean} are performed for the curvature in the classical sense, that is when the connection form \(\omega\) is continuously differentiable. One could naturally ask whether the conclusion of \cref{mainTheorem} holds when \(\Omega\) is merely defined in a weak sense \cite{Uhlenbeck_1982} but still continuous, or whether \cref{corollaryEuclidean} holds when the weak curvature is bounded. If \(\sigma\) is a regular parametrization of a surface in \cref{mainTheorem} we can consider the question about suitable traces of the curvature that make the formula valid.

\section{Preliminaries}

\subsection{Properties of the amplitude of holonomy along paths}
We state here some useful properties on the amplitude of holonomies along paths.

\begin{proposition}[Amplitude of concatenated holonomies]
If the metric on \(G\) is left-invariant, then for every
\(\gamma \in C^1([0, 1], M)\) and \(\eta \in C^1([0, 1], M)\) and if \(\gamma (1) = \eta (0)\), then
\begin{equation*}
 \ampl{\gamma \cdot \eta} \le \ampl{\gamma} + \ampl{\eta}
 \eofs.
\end{equation*}
\end{proposition}
\begin{proof}
We have by definition of the concatenation
\begin{equation*}
  \gamma \cdot \eta (t)
  =
  \begin{cases}
   \gamma (2t) &\text{if \(t \in [0, \frac{1}{2}]\)},\\
   \eta (2 t - 1) & \text{if \(t \in [\frac{1}{2}, 1]\)}.
  \end{cases}
\end{equation*}
We next observe that \(\operatorname{Pt}_{\gamma \cdot \eta} (t) = \operatorname{Pt}_{\gamma} (2t)\) if \(t \in [0, \frac{1}{2}]\) and \(\operatorname{Pt}_{\gamma \cdot \eta} (t) = \operatorname{Pt}_{\eta (2t - 1)} \compose \operatorname{Pt}_{\gamma (1)}\) if \(t \in [\frac{1}{2}, 1]\).
It follows that if \(g \in C^1 ([0, 1], G)\) and \(h \in C^1([0, 1], G)\) are homotopic to \(\operatorname{Pt}_\gamma\) and \(\operatorname{Pt}_\eta\) relatively to \(\{0, 1\}\), then the map \(f : C^1 ([0, 1], G)\) defined by
\begin{equation*}
f (t)
  \defeq
  \begin{cases}
   g (2t) &\text{if \(t \in [0, \frac{1}{2}]\)},\\
   h (2 t - 1)\, g (1) & \text{if \(t \in [\frac{1}{2}, 1]\)},
  \end{cases}
\end{equation*}
is homotopic to \(\operatorname{Pt}_{\gamma \cdot \eta}\) and the conclusion thus follows by right-invariance of the metric on \(G\).
\end{proof}

\begin{proposition}[Amplitude of conjugate holonomy]
If the metric on \(G\) is right-invariant, then for every
\(\gamma \in C^1 ([0, 1], M)\) and every \(\eta \in C^1 ([0, 1], M)\) such that \(\eta (0) = \gamma (0) = \gamma (1)\), one has
\begin{equation*}
 \ampl{\Bar{\eta} \cdot \gamma \cdot \eta} = \ampl{\gamma}
 \eofs.
\end{equation*}
\end{proposition}
\begin{proof}
  Assume that \(g \in C^1 ([0, 1], G)\) is homotopic to \(\operatorname{Pt}_\gamma\) relatively to \(\{0, 1\}\) and that \(h \in C^1 ([0, 1], G)\) is homotopic to \(\operatorname{Pt}_\eta\) relatively to \(\{0, 1\}\).
We construct the map \(H : [0, 1] \times [0, 1] \to M\) by setting for \((s, t) \in [0, 1] \times [0, 1]\),
\begin{equation*}
 H (s, t)
 \defeq
 \begin{cases}
      h (1 - 3s  - t) \,h (1 - t)^{-1}    &\text{if \(0 \le s \le \frac{1 - t}{3}\)},\\
     g (\frac{3 s + t - 1}{1 + 2t}) \,h (1) \,h (1 - t)^{-1}  &\text{if \(\frac{1 -t}{3} \le s \le \frac{2 + t}{3}\)},\\
     h (3s - 2 - t)\, g (1)\, h (0)\, h (1 - t)^{-1} & \text{if \(\frac{2 + t}{3} \le s \le 1\)}.
 \end{cases}
\end{equation*}
We conclude thus that \(g\) is homotopic to \(\operatorname{Pt}_{\Bar{\eta} \cdot \gamma \cdot \eta}\) and thus the conclusion follows.
\end{proof}

\subsection{Axial gauge}

Our analysis will be facilitated by working with a trivialization that corresponds to the \emph{axial gauge}, also known as \emph{Arnowitt--Fickler gauge} \citelist{\cite{Itzykson_Zuber_1980}*{12-1-1}\cite{Arnowitt_Fickler_1962}}.

\begin{proposition}
\label{proposition_axial_gauge}
For every pont \(p \in M\) and every \(v \in \Rset^m\), there exists a local trivialization \(U \times F\) such \(p \in U\) and \(v \intprod \omega= 0\) everywhere in \(U\) in this local trivialization.
\end{proposition}
Here \(v \intprod \omega\) denotes the \emph{interior multiplication} (or \emph{contraction}) of the form \(\omega\) by the vector \(v\):
\(
    v \intprod \omega (x)
  \defeq
    \omega (x)[v] \in \mathfrak{g}
\), which is also denoted by \(i_v \omega\).

When \(M = \Rset^m\), \(F = \Cset\), \(G = U (1)\) and \(v  \in \Rset^m\) is a fixed vector,
then the connection form \(\omega\) can be described by setting \(\omega (w) = i A \cdot w\) for some vector field \(A : \Rset^n \to \Rset^n\) and for every \(w \in \Rset^m\), and then the axial gauge prescribes that the component \(A \cdot v\) of the vector field \(A\) vanishes everywhere.
The axial gauge does not fix the curvature form in directions transversal to \(v\).

\begin{proof}[Proof of \cref{proposition_axial_gauge}]
Let \(\Tilde{\Phi} : U \times F \to B\) be a local trivialization of the bundle \(B\) and \(U\) is a ball. That is \(\Phi\) is a diffeomorphism and \(\Tilde{\Phi}\) is linear on each fiber.
Let \(\Tilde{\omega}\) be the connection form on \(U\).
We define now a function \(g : U \to G\) by the condition that \(v \intprod (dg + \Tilde{\omega} g) = (dg + \Tilde{\omega} g)[v] = 0\). This can be done by parallel transport on every straight line parallel to the vector \(v\).
We conclude by considering the map
\(\Phi \defeq \Tilde{\Phi} \compose g\).
\end{proof}

\section{Derivative of the holonomy}

We define for \(r > 0\), the path \(\gamma_r : [0, 1] \to \Rset^2\) for each \(t \in [0, 1]\) by \(\gamma_r (t) \defeq (r\cos 2 \pi t, r \sin 2 \pi t)\).
We compute the holonomy on a circle of radius \(r > 0\) by finding a function \(g \in C^1 ([0,1], G)\) that satisfies the equation
\begin{equation}
\label{eq_rae2Yeengo}
 \left\{
 \begin{aligned}
   g'_r (t) + 2 \pi \, r \omega (r e^{2 \pi i t})[i e^{2 \pi i t}] g_r (t) &= 0
   & & \text{for \(t \in [0, 1]\)},\\
   g_r (0) & = \operatorname{id},
 \end{aligned}
 \right.
\end{equation}
where the plane \(\Rset^2\) is identified with the field of complex numbers,
so that \(e^{2 \pi i t} = (\cos 2 \pi t, \sin 2 \pi t)\) and
\(i e^{2 \pi i t} = (-\sin 2 \pi t, \cos 2 \pi t)\).
The holononomy at \((r, 0)\) is then given by \(g_r (1)\).

The core of the proof of \cref{mainTheorem} lies in the following derivative formula.

\begin{lemma}
If \(B_R \subset \Rset^2\) and if \(B_R \times F\) is a vector bundle, then for each \(r \in (0, R)\) one has
\begin{equation*}
    \frac{d}{dr} g_r (1) - \omega (r)[1] \,g_r (1) + g_r (1)\, \omega (r) [1]
  =
    2 \pi \, r
    \int_0^{1} g_r (1) \,g_r (t)^{-1} \, \Omega (r e^{2 \pi i t})[e^{2 \pi i t}, i e^{2 \pi i t}] \,g_r (t)\dif  t
    \eofs .
\end{equation*}

\end{lemma}
\begin{proof}
We define \(h_r (s)\defeq \frac{\partial}{\partial r} g_r (s)\).
In view of the holonomy equation \eqref{eq_rae2Yeengo}, the function \(h_r\) satisfies the system
\begin{equation*}
\left\{
 \begin{aligned}
    \frac{h_r' (t)}{2 \pi}
  +
     r \omega (r e^{2 \pi i t})[i e^{2 \pi i t}] \, h_r (t)
    &=
      - (
    \omega (r e^{2 \pi i t})
      + r \partial_r \omega (r e^{2 \pi i t}))[i e^{2 \pi i t}]
      \,g_r (t)
      & & \text{for \(t \in [0, 1]\)},\\
    h_r (0) & = 0
    \eofs.
\end{aligned}
\right.
\end{equation*}

By variation of parameters for solutions of differential equations (see for example \cite{Hartmam_1964}*{Corollary 2.1}), we have for each \(r \in (0, R)\),
\begin{equation*}
  h_r (1)
  =
  - 2 \pi \int_0^{1}
    g_r (1) \,g_r (t)^{-1} \,
    \biggl(\omega (r e^{2 \pi i t}) + r \, \frac{\mathrm{d}}{\mathrm{d}r} \omega (r e^{2 \pi i t})\biggr)[i e^{2 \pi i t}]  \, g_r (t) \dif  t
    \eofs.
\end{equation*}
We note that
\begin{multline*}
 2 \pi \int_0^{1}g_r (1) \, g_r (t)^{-1} \, \omega (r e^{2 \pi i t})[i e^{2 \pi i t}] \, g_r (t) \dif  t\\
 = \int_0^{1} g_r (1) \, g_r (t)^{-1} \, \omega (r e^{2 \pi i t})\biggl[\frac{\mathrm{d}}{\mathrm{d}t} \bigl(e^{2 \pi i t}\bigr)\biggr] \, g_r (t) \dif  t.
 \end{multline*}
Integrating by parts the term on the right-hand side we have,
\begin{equation*}
\begin{split}
 2 \pi \int_0^{1}g_r (1) \, g_r (t)^{-1} &\, \omega (r e^{2 \pi i t})[i e^{2 \pi i t}] \, g_r (t) \dif  t
 \\
 = &\,\omega (r)[1] \, g_r (1) - g_r (1) \, \omega (r) [1]\\
 &-2 \pi \int_0^{1} g_r (1)\, g_r (t)^{-1}\, \omega (r e^{2 \pi i t})[r i e^{2 \pi i t}]
 \,
 \omega (r e^{2 \pi i t})[e^{2 \pi i t}]
 \,
 g_r (t) \dif  t\\
 &- 2 \pi \int_0^{1} g_r (1) \, g_r (t)^{-1} \, \biggl(\frac{\mathrm{d}}{\mathrm{d}t} \omega (r e^{2 \pi i t})\biggr)[e^{2 \pi i t}]
 \,
 \omega (r e^{2 \pi i t})[e^{2 \pi i t}]
 \,
 g_r (t) \dif  t\\
  & + 2 \pi \int_0^{1}
  g_r (1) \, g_r (t)^{-1} \, \omega (r e^{2 \pi i t})[e^{2 \pi i t}]\,
 \omega (r e^{2 \pi i t})[r i e^{2 \pi i t}] \, g_r (t) \dif  t.
\end{split}
\end{equation*}
We conclude that
\vspace{1ex}
\begin{equation*}
\begin{split}
  h_r &(1)\\[-1ex]
  &
  =
  - 2 \pi \, r \int_0^{1}g_r (1) \, g_r (t)^{-1} \, \bigl(d \omega (r e^{2 \pi i t}) + \omega (r e^{2 \pi i t}) \wedge \omega (r e^{2 \pi i t})\bigr) [e^{2 \pi i t}, i e^{2 \pi i t}] \, g_r (t) \dif  t
  \\
  &=
  -2 \pi \, r \int_0^{1} g_r (1) \, g_r (t)^{-1}\, \Omega (r e^{2 \pi i t})[e^{2 \pi i t}, i e^{2 \pi i t}] \, g_r (t)\dif  t
  \eofs.
  \qedhere
\end{split}
\end{equation*}
\end{proof}

By placing ourselves, in view of \cref{proposition_axial_gauge}, in an axial gauge with respect to the vector \(v = (0, 1) \in \Rset^2 \), we obtain the formula
\begin{equation*}
    \frac{\mathrm{d}}{\mathrm{d}r} g_r (1)
  =
    2 \pi \, r \int_0^{1} g_r (1) \, g_r (t)^{-1}\, \Omega (r e^{2 \pi i t})[e^{2 \pi i t}, i e^{2 \pi i t}] g_r (t)\dif  t
    \eofs .
\end{equation*}
and we deduce since the metric is bi-invariant,

\begin{proposition}
\label{proposition_amplitude_derivative}
If \(B_R \subset \Rset^2\) and if \(B_R \times F\) is a vector bundle, then for every \(r \in (0, R)\),
\begin{equation*}
  \limsup_{s \to r} \frac{\abs{\ampl{\gamma_r}- \ampl{\gamma_s}}}{\abs{r - s}}
  \le r \int_{\Sset^1} \abs{\Omega (r e^{i\theta})}.
\end{equation*}
\end{proposition}

We then obtain as a consequence.

\begin{proposition}
\label{proposition_amplitude_derivative}
If \(B_R \subset \Rset^2\) and if \(B_R \times F\) is a vector bundle, then
\begin{equation*}
    \ampl{\gamma_R}
  \le
    \int_{B_R} \abs{\Omega}.
\end{equation*}
\end{proposition}

\Cref{mainTheorem}
follows then by applying a pull-back to the curvature.

In the framework of weak connections, a natural generalization of \cref{proposition_amplitude_derivative} would be the case where \(\omega \in W^{1, 4} (B_R)\) so that \(\Omega \in L^1 (B_R)\) \cite{Uhlenbeck_1982}*{Lemma 1.1}.

\end{document}